\documentclass[a4paper,12pt]{article}
\usepackage{amssymb}                                    
\usepackage{mathrsfs}                    
\usepackage{amsmath}                    
\usepackage{amsfonts}                   
\usepackage{amsthm}                     
\usepackage[latin1]{inputenc}           
\usepackage[arrow, matrix, curve]{xy}    
\usepackage[english]{babel}                
\usepackage{bibgerm}				 
\usepackage{graphics}
\usepackage{hyperref}

\theoremstyle{plain}
\newtheorem{thm}{Theorem}[section]
\newtheorem{corollary}[thm]{Corollary}
\newtheorem{proposition}[thm]{Proposition}
\newtheorem{lemma}[thm]{Lemma}

\newtheorem{remark}[thm]{Remark}
\newtheorem{prop}[thm]{Proposition}

\theoremstyle{definition}
\newtheorem{definition}[thm]{Definition}

\def\sSet        {\text{sSet}} 
\def\Top         {\text{Top}} 
\def\Alg         {\text{AlgKan} }
\def\AlgC        {\text{Alg}\mathcal{C}}
\def\AlgCH       {\texorpdfstring{$\AlgC$}{AlgC}}

\def\colim       {\text{colim}}

\def \dotto      {\!\xymatrix@C=0.5cm{\ar@{-->}[r] &}}
\def\indlim      {\underrightarrow{\lim}}
\def\CoEq        {\text{CoEq}}
\def\Sing        {\text{Sing}}  
\def\ASing       {\Pi_\infty}  
\def\C	         {\mathcal{C}}
\def\AlgQ		 {\text{AlgQuasi}}
\def\inftyH      {\texorpdfstring{$\infty$}{oo}}

\begin{document}

\title{Algebraic models for higher categories}
\author{Thomas Nikolaus \thanks{Email address: Thomas.Nikolaus@uni-hamburg.de \newline 2000 Mathematics Subject Classification: Primary: 55U35; Secondary: 18G30.} \\ \\
Fachbereich Mathematik, \ Universit\"at Hamburg\\
  Schwerpunkt Algebra und Zahlentheorie\\
  Bundesstra\ss e 55, \ D\,--\,20\,146\, Hamburg}
\date{}
\maketitle




\begin{abstract} \noindent
We introduce the notion of algebraic fibrant objects in a general model category and establish a (combinatorial) model category structure on algebraic fibrant objects. Based on this construction we propose algebraic Kan complexes as an algebraic model for $\infty$-groupoids and algebraic quasi-categories as an algebraic model for $(\infty,1)$-categories. We furthermore give an explicit proof of the homotopy hypothesis.
\end{abstract}

\tableofcontents
\section{Introduction}

Simplicial sets have been introduced as a combinatorial model for topological spaces. It has been known for a long time that topological spaces and certain simplicial sets called Kan complexes are 'the same' from the viewpoint of homotopy theory. To make this statement precise Quillen \cite{quillen1967homotopical} introduced the concept of model categories and equivalence of model categories as an abstract framework for homotopy theory. He endowed the category $\Top$ of topological spaces and the category $\sSet$ of simplicial sets with model category structures and showed that $\Top$ and $\sSet$ are equivalent in his sense. He could identify Kan complexes as fibrant objects in the model structure on $\sSet$. \\

Later higher category theory came up. A 2-category has not only objects and morphisms, like an ordinary (1-)category, but also 2-morphisms, which are morphisms between morphisms. A 3-category has also 3-morphisms between 2-morphisms and so on. Finally an $\infty$-category has $n$-morphisms for all $n \geq 1$. Unfortunately it is very hard to give a tractable definition of $\infty$-categories. See \cite{leinster2002survey} for several definitions of higher categories. 

An interesting subclass of all $\infty$-categories are those $\infty$-categories for which all $n$-morphisms are invertible. They are called $\infty$-groupoids. A standard construction from algebraic topology is the fundamental groupoid construction $\Pi_1(X)$ of a topological space $X$. Allowing higher paths in $X$ (i.e. homotopies) extends this construction to a fundamental $\infty$-groupoid $\Pi_\infty(X)$. It is widely believed that every $\infty$-groupoid is, up to equivalence, of this form. This belief is called the homotopy hypothesis \cite{BaezHom}.

There is another important subclass of $\infty$-categories, called $(\infty,1)$-categories. These are $\infty$-categories where all $n$-morphisms for $n \geq 2$ are invertible. Thus the only difference to $\infty$-groupoids is that there may be non-invertible $1$-morphisms. In particular the collection of all small $\infty$-groupoids forms a $(\infty,1)$-category. Another example of a $(\infty,1)$-category is the category of topological spaces where the $n$-morphisms are given by $n$-homotopies. In the language of $(\infty,1)$-categories a more refined version of the homotopy hypothesis is the assertion that the fundamental $\infty$-groupoid construction provides an equivalence of the respective $(\infty,1)$-categories. \\

From the perspective of higher category theory Quillen model structures are really presentations of $(\infty,1)$-categories, see e.g. \cite{lurie2006higher} appendix A.2 and A.3. Hence we think about a model category structure as a generators and relations description of a $(\infty,1)$-category. A Quillen equivalence then becomes an adjoint equivalence of the presented $(\infty,1)$-categories. Thus the classical Quillen equivalence between topological spaces and simplicial sets really encodes an equivalence of $(\infty,1)$-categories. 

Keeping this statement in mind it is reasonable to think of a simplicial set $S$ as a model for an $\infty$-groupoid. The $n$-morphisms are then the $n$ simplices $S_n$. And in fact there has been much progress in higher category theory using simplicial sets as a model for $\infty$-groupoids. This model has certain disadvantages. First of all a simplicial set does not encode how to compose $n$-morphisms. But such a composition is inevitable for higher categories. This problem is usually addressed as follows:

The model structure axioms on $\sSet$ imply that in the corresponding $(\infty,1)$ category each simplicial set is equivalent to a fibrant object i.e.\ a Kan complex. It is possible to interpret the lifting properties of a Kan complex $S$ as the existence of compositions in the $\infty$-groupoid, see section \ref{sec:kan}. Although the lifting conditions ensure the existence of compositions for $S$, these compositions are only unique up to homotopy. This makes it sometimes hard to work with Kan complexes as a model for $\infty$-groupoids. Another disadvantage is that the subcategory of Kan complexes is not very well behaved, for example it does not have colimits and is not locally presentable. \\

The idea of this paper to solve these problems is to consider a more algebraic version of Kan complexes as model for $\infty$-groupoids. More precisely we will consider Kan complexes endowed with the additional structure of distinguished fillers. We call them algebraic Kan complexes. We show that the category of algebraic Kan complexes has all colimits and limits and is locally presentable(theorem \ref{algKanMain}.3). Furthermore we endow it with a model structure and show that it is Quillen equivalent to simplicial sets (theorem \ref{algKanMain}.4-.5) . The name algebraic will be justified by identifying algebraic Kan complexes as algebras for a certain monad on simplicial sets (theorem \ref{algKanMain} .1-2). The fact that algebraic Kan complexes really model $\infty$-groupoids will be justified by a proof of the appropriate version of the homotopy hypothesis (corollary \ref{homhyp}).

We will generalize this notion of algebraic Kan complex to algebraic fibrant objects in a general model category $\C$, which satisfies some technical conditions, stated at the beginning of section \ref{allgemein}. In particular we show that $\C$ is Quillen equivalent to the model category $\AlgC$ of algebraic fibrant objects in $\C$ (theorem \ref{main}). We show that $\AlgC$ is monadic over $\C$ (proposition \ref{sec:monad}) and that all objects are fibrant. In addition we give a formula how to compute (co)limits in $\AlgC$ from (co)limits in $\C$ (section \ref{sec:colimits}). 

Finally we apply the general construction to the Joyal model structure on $\sSet$. This is a simplicial model for $(\infty,1)$-categories \cite{JoyalBarcelona, lurie2006higher}. The fibrant objects are called quasi-categories. We propose the category $\AlgQ$ of algebraic quasi-categories as our model for $(\infty,1)$-categories (section \ref{algqdef}). One of its major advantages is that the model structure on algebraic quasi-categories can be described very explicitly, in particular we will give sets of generating cofibrations and trivial cofibrations (theorem \ref{AlgQMain}). Such a generating set is not known for the Joyal structure. \\

There have been other proposals for algebraic definitions of $\infty$-groupoids, see e.g. \cite{Cis07}. But as pointed out in the introduction of \cite{Mal} the issue to find a locally presentable algebraic model is still open. Strictly speaking the problem described there is solved by the model of algebraic Kan complexes. Nevertheless higher category theorists may feel that algebraic Kan complexes are after all not the kind of model that they were looking for. But since the model presented here is formally algebraic by the usual definition, this at least indicates that the formal definition for what is supposed to count as an algebraic model for $\infty$-groupoids might need to be refined.\\

This paper is organized as follows. In Section \ref{allgemein} we give the definition of the category $\AlgC$ for a general model category $\C$. We prove that $\AlgC$ enjoys excellent categorical properties and that it admits a model structure Quillen equivalent to $\C$. In Section \ref{algKan2} we investigate algebraic Kan complexes as a model for $\infty$-groupoids. In particular we prove the homotopy hypothesis. In section \ref{quasi} we investigate algebraic quasi-categories as a model for $(\infty,1)$-categories. Furthermore we compare algebraic Kan complexes and algebraic quasi-categories. In Section \ref{outlook} we sketch some further possible applications.\\

{\bf Acknowledgements.} The author is supported by the Collaborative Research Centre 676 ``Particles, Strings and the Early Universe - the Structure of Matter and Space-Time''. The author thanks Urs Schreiber, Mike Shulman, Emily Riehl, Richard Garner, Till Barmeier and Christoph Schweigert for helpful discussions and comments on the draft.

\section{Algebraic fibrant objects}\label{allgemein}

Let $\C$ be a cofibrantly generated model category. For the terminology of model categories we refer to \cite{hovey2007model}.
Furthermore we make the assumption:
\begin{center}
  All trivial cofibrations in $\C$ are monic. 
\end{center}
This is true in many model categories. For example in simplicial sets with either the Quillen or the Joyal model structure. \\

Choose a set of trivial cofibrations
$$ \{ A_j \to B_j \}_{j \in J}$$
in $\C$ such that an object $X\in \C$ is fibrant iff for every morphism $A_j \to X$ with $j\in J$ there exists a filler, that means a morphism $B_j \to X$ rendering the diagram
\begin{equation} \label{filler}
\xymatrix{
A_j \ar[r]\ar[d] & X \\
B_j \ar[ru] & 
}
\end{equation}
commutative. We could take $J$ to be a set of generating trivial cofibrations but in general $J$ might be smaller. We assume that the domains $A_j$ are small objects and that $\C$ is cocomplete, so that
Quillen's small object argument yields a fibrant replacement. For simplicity we assume that the $A_j$'s are $\omega$-small but everything is still valid if they are only $\kappa$-small for an arbitrary (small) regular cardinal $\kappa$.\\

In order to have a more algebraic model for fibrant objects we want to fix fillers for all diagrams.

\begin{definition}
An \emph{algebraic fibrant object} (of $\C$) is an object $X \in \C$ together with a distinguished filler for each morphism $h: A_j \to X$ with $j \in J$. That means a morphism $F(h): B_j \to X$ rendering diagram \eqref{filler} commutative. A map of algebraic fibrant objects is a map that sends distinguished fillers to distinguished fillers. The category of algebraic fibrant objects is denoted by $\AlgC$.
\end{definition}

In particular for each algebraic fibrant object the underlying object $X \in \C$ is fibrant because all fillers exist. Now we have the canonical forgetful functor 

$$ U: \AlgC \to \C $$
 
which sends an algebraic fibrant object to the underlying object of $\C$. The task of this section is to show that $U$ induces an equivalence between model categories. More precisely, we want to endow $\AlgC$ with a model category structure and show that $U$ is the right adjoint of a Quillen equivalence.

\subsection{Free algebraic fibrant objects}\label{sec:free}

As a first step we give an explicit description of the left adjoint 
$$F: \C \to \AlgC$$
called the free algebraic fibrant object functor. \\

We want to use Richard Garner's improved version of Quillen's small object argument \cite{garner2009understanding}. 
The idea is to start with an object $X \in \C$ and to successively add  fillers in a free way. So we define $X_1$ to be the pushout (in $\C$) of the diagram
$$
\xymatrix{
\bigsqcup A_j \ar[r] \ar[d] & X \ar[d] \\
\bigsqcup B_j \ar[r] & X_1
}
$$
where the disjoint unions are taken over all $j \in J$ and morphisms $A_j \to X$. Note that the inclusion $X \to X_1$ is monic due to our assumptions that trivial cofibrations are monic. \\

For those morphisms $h: A_j \to X_1$ which factor through $X$ the structure morphisms $B_j \to X_1$ provide fillers. These are well defined because the factorization of $h$ through $X$ is unique since $X \to X_1$ is monic. Unfortunately there might be morphisms $h: A_j \to X_1$ which do not factor through $X$. Thus in order to add additional fillers for these let $X_2$ be the pushout

$$
\xymatrix{
\bigsqcup A_j \ar[r] \ar[d] & X_1 \ar[d] \\
\bigsqcup B_j \ar[r] & X_2
}
$$
where the disjoint union is taken over $j \in J$ and those morphisms $A_j \to X_1$ which do not factor through $X$. Note that this differs from the ordinary small object argument where this colimit is taken over all morphisms $A_j \to X_1$. We again bookmark the fillers $B_j \to X_2$ and proceed inductively. 
Eventually we obtain a sequence 
$$ X  \to X_1 \to X_2 \to X_3 \ldots $$
where all morphisms are by construction trivial cofibrations. Let 
$$X_\infty := \indlim(X  \to X_1 \to X_2 \to X_3 \ldots)$$
be the colimit over this diagram. Note that the inclusion $X \to X_\infty$ is by construction a trivial cofibration, 
in particular a weak equivalence.\\

Now let $h: A_j \to X_\infty$ be a morphism. Because $A_j$ is $\omega$-small this factors through a finite step and our construction implies that there is a unique smallest $m$ such that $h$ factors through $X_m$. Then we have the filler $$F(h): B_j \to X_{m+1} \to X_{\infty}.$$
This makes $X_\infty$ into an algebraic fibrant object. 

\begin{prop}\label{adjoint}
The assignment $F: \C \to \AlgC$ which sends $X$ to $X_\infty$ is left adjoint to $U: \AlgC \to C$. Furthermore the unit of this adjunction is the inclusion $X \to X_\infty$, hence a weak equivalence. 
\end{prop}

\begin{proof}
Let $Z$ be an algebraic fibrant object. We have to show that for a morphism $\varphi: X \to U(Z)$ in $\C$ there is a unique morphism $\varphi_\infty: X_\infty \to Z$ in $\AlgC$ rendering
$$
\xymatrix{
X \ar[r]^\varphi\ar[d] & Z \\
X_\infty \ar@{-->}[ru]_{\exists ! \varphi_\infty} & 
}
$$
commutative.
But this is trivial since all we did in going from $X$ to $X_\infty$ was gluing new distinguished fillers which have to be sent to the distinguished fillers in $Z$. 
\end{proof}

\begin{corollary}\label{retrakt}
Each of the maps $i: FA_j \to FB_j$ admits a canonical retract (left inverse).
\end{corollary}
\begin{proof}
A retract is a map $r: FB_j \to FA_j$ such that the composition 
$$FA_j \stackrel{i}{\to} FB_j \stackrel{r}{\to} FA_j$$
is the identity. Because $F$ is left adjoint to $U$, this is the same thing as a map $r':B_j \to UFA_j$ such that
the composition 
$$ A_j \to B_j \stackrel{r'}{\to} UFA_j$$
is the unit of the adjunction, i.e. the inclusion $i': A_j	\to FA_j$. Hence $r'$ is just a filler for the morphism $i'$, and such a filler exists canonically because $FA_j$ is an algebraic fibrant object.
\end{proof}

\subsection{Monadicity}\label{sec:monad}

In this section we want to show that algebraic fibrant objects in $\C$ are algebras for a certain monad. This is a rather direct justification to call them algebraic. Let $T$ be the monad which is induced by the adjunction
$$ 
\xymatrix{
F : \C \ar@<0.3ex>[r] & \AlgC : U \ar@<0.7ex>[l]
}$$
That means $T = U \circ F : \C \to \C$. 

\begin{prop}\label{monad}
The category $\C^T$ of $T$-algebras in $\C$ is equivalent to the category $\AlgC$. More precisely the functor $U$ induces
an equivalence $U^T: \AlgC \to \C^T$.  
\end{prop}

In abstract language the proposition states that the adjunction $(F,U)$ is monadic. By Beck's monadicity theorem we have to show that
\begin{enumerate}
\item\label{one} a morphism $f$ in $\AlgC$ is an isomorphism iff $U(f)$ is an isomorphism in $\C$;
\item\label{two} $\AlgC$ has coequalizers of $U$-split coequalizer pairs and $U$ preserves those coequalizers.
\end{enumerate}

Let us first turn towards property \ref{one}. Assume $f: X \to Y$ is a morphism in $\AlgC$ such that $U(f)$ is an isomorphism in $C$ with inverse $g: U(Y) \to U(X)$. It suffices to show that $g$ is a morphism in $\AlgC$, i.e.  it sends distinguished fillers to distinguished fillers. But this is satisfied since $f$ and $g$ induce isomorphisms between sets 
$\hom_{\C}(A_j,X) \cong \hom_{\C}(A_j,Y)$ and $\hom_{\C}(B_j,X) \cong \hom_{\C}(B_j,Y)$ and thus $g$ preserves distinguished fillers since $f$ does. \\

The second property is seemingly more involved. A parallel pair of arrows $f,g: X \to Y$ in $\AlgC$ is a $U$ split coequalizer pair if the corresponding coequalizer diagram in $\C$
$$\xymatrix{
 U(X) \ar@<0.7ex>[r]^-{U(g)}\ar@<-0.3ex>[r]_-{U(f)} & U(Y) \ar[r]^-{\pi}& Q
}$$
allows sections $s$ of $\pi$ and $t$ of $U(f)$ such that $U(g) \circ t = s \circ \pi$. We will endow $Q$ with the structure of an algebraic fibrant object such that it is the coequalizer of the initial pair $f,g$ in $\AlgC$. Therefore we have to fix a filler $F(h): B_j \to Q$ for each morphism $h: A_j \to Q$. Since $s$ is a section of $\pi$, the image of the morphism
$$ s \circ h : A_j \to Y $$
under $s$ is $h$. Thus we let $F(h)$ be the image of the distinguished filler for $s \circ h$. Then the following lemma shows that property \ref{two} and thus proposition \ref{monad} holds. 
\begin{lemma}
The morphism $\pi: Y \to Q$ is a morphism in $\AlgC$ which is a coequalizer of the pair $f,g: X \to Y$.
\end{lemma}
\begin{proof}
First we check that $\pi$ lies in $\AlgC$.
 Take a morphism $h: A_j \to Y$. By definition of fillers in $Q$ we have to show that the fillers 
$F(h),F(\tilde h) : B_j \to Y$ for 
$h: A_j \to Y$ and for $\tilde h := s \circ \pi \circ h: A_j \to Y$ are sent to the same filler in $Q$. But we have 
$h = U(f) \circ t \circ h$ and $\tilde h = U(g) \circ t \circ h$ and therefore this follows from the fact that $Q$ is the coequalizer.

Now we want to verify the universal property. We have to check that for each morphism $\varphi: Y \to Z$ in $\AlgC$, such that $\varphi \circ f = \varphi \circ g: X \to Z$ there is a unique $\varphi_Q: Q \to Z$ in $\AlgC$ such that 
$$ \xymatrix{
X \ar@<0.7ex>[r]^-{g}\ar@<-0.3ex>[r]_-{f} & Y \ar[r]^-{\pi}\ar[d]_{\varphi}& Q\ar@{-->}[ld]^{\exists !\varphi_Q} \\
& Z & 
}$$
commutes. From the fact that $Q$ is the coequalizer in $\C$ we obtain a unique morphism $\varphi_Q$ and it only remains to show that it lies in $\AlgC$, i.e. preserves distinguished fillers. But this is automatic since $\varphi$ preserves distinguished fillers and fillers in $Q$ are by definition images of distinguished fillers in $Y$.
\end{proof}

\subsection{An auxiliary construction}\label{sec:aux}

In this section we want to prove propositions \ref{fact} and \ref{fact2}, which we will use to investigate colimits in the next section. The impatient reader can skip this section and just take note of these propositions. 

\begin{prop}\label{fact}
Let $Y$ be an algebraic fibrant object, $X$ be an object in $\C$ and 
$$ f: Y \to X $$
be a morphism (that means a morphism $f:U(Y) \to X$ in $\C$). Then there is an algebraic fibrant object $X^{f}_{\infty}$ together with a morphism $X \to X_\infty^f$ such that the composite 
$$ Y \to X \to X^{f}_{\infty}$$
is a map of algebraic fibrant objects. \\ 
Furthermore $X^{f}_{\infty}$ is initial with this property. That means it satisfies the following universal property: For each morphism $\varphi: X \to Z$ in $\C$ where $Z$ is an algebraic fibrant object, such that the composite $Y \to X \to Z$ is a morphism of algebraic fibrant objects there exists a unique morphism $\varphi_\infty: X^{f}_{\infty} \to Z$ rendering the diagram
$$
\xymatrix{
Y \ar[d]_f \ar[rd]^{\text{alg.}} & \\
X \ar[d] \ar[r]^{\varphi} & Z \\
X^{f}_{\infty} \ar@{-->}[ru]_{\exists !  \varphi_\infty} & 
}
$$
commutative. \\
If $f$ is monic in $\C$, then $X \to X_\infty^f$ is a trivial cofibration.
\end{prop}

Before we prove this proposition we first draw some conclusions.

\begin{remark}
\begin{itemize}
\item
Let $\emptyset$ be the initial algebraic fibrant object.  Assume that the underlying object $U(\emptyset)$ is initial in $\C$. In particular we have a unique morphism $f: U(\emptyset) \to X$. Then the universal property of $X_\infty^f$ reduces to the universal property of $X_\infty$ from proposition \ref{adjoint}. We will use this observation to give a very similar construction of $X_\infty^f$ for arbitrary $f$. Note that the assumption that $U(\emptyset)$ is initial is not always satisfied, but it holds in the case of algebraic Kan complexes and algebraic quasi-categories. 
\item In contrast to the map $X \to X_\infty$ in the case $Y = \emptyset$ the morphism $X \to X_\infty^f$ is in general neither a cofibration nor a weak equivalence. Nevertheless if $f: Y \to X$ is a monomorphism, the proposition says that it is still a trivial cofibration.
\end{itemize}
\end{remark}

We will construct $X^f_\infty$ in two steps. First consider the images under $f$ of distinguished filler diagrams of $Y$. These are diagrams
$$
\xymatrix{
A_j \ar[r]\ar[d]&  Y \ar[r] & X \\
B_j \ar[ru] \ar[rru] & &
}
$$
If we want to turn $f$ into a morphism of algebraic fibrant objects, these diagrams have to be distinguished filler diagrams of $X$. But then it might occur that a morphism $A_j \to X$ factors in different ways through $Y$ and provides different fillers $B_j \to X$. In order to avoid this ambiguity we want to identify them. Note that this ambiguity does not occur if $f$ is a monomorphism, so in this case we can skip this next step. \\

Let us describe this more technically: Let $H$ be the set of morphisms $h: A_j \to X$ which factor through $Y$. For each $h \in H$ let  $F_h$ be the set of fillers $\varphi: B_j \to X$ which are the images of distinguished fillers. $F_h$ has at least one element but it might of course be an infinite set. Now fix an element $h_0 \in H$. In order to identify the different fillers $F_{h_0}$ we take the coequalizer
$$ X_{h_0} := \CoEq\big( B_j \stackrel{\varphi}{\to} X | \varphi \in F_{h_0}\big )$$
which comes with a morphism $p_{h_0}: X \to X_{h_0}$. We now want to repeat this process inductively in order to identify the fillers for all horns $H$. Therefore we endow $H$ with a well-ordering such that $h_0$ is the least element. Assume that $X_{h'}$ is defined for all $h' < h$. Let 
$$X_{< h} := \indlim(X_{h'} \mid h' < h) $$
and define 
$$ X_h := \CoEq\big( B_j \stackrel{\varphi}{\to} X \to X_{< h} | \varphi \in F_{h}\big )$$
to be the object in $\C$ where the fillers in $F_{h}$ are identified. In this coequalizer several morphisms $B_j \stackrel{\varphi}{\to} X \to X_{< h}$ might occur, which are equal but they only contribute once (hence all copies could be left out). 
Finally let $X_H$ be the colimit 
$$ X_H := \indlim(X_h \mid h \in H) $$
and $p: X \to X_H$ be the inclusion.

\begin{remark}
We could also describe the whole process by a single colimit over a diagram $D$. For $D$ we take the diagram with an object $\left(B_j\right)_h$ for each morphism $h \in H$ (the index $h$ is just for bookkeeping) and the object $X$. Morphisms in this diagram are all $\varphi: \left(B_j\right)_h \to X$ with $\varphi \in F_h$ and no further morphisms. Then
$$ X_H = \colim D$$
which is seen by verifying the universal property.
\end{remark}

Now for each $h: A_j \to X$ which factors through $Y$ the possible distinguished fillers are all identified in $X_H$. But unfortunately a morphism $A_j \to X_H$ which factors through $Y$ might factor through $X$ in different ways and thus lead to different morphisms $A_j \to X$. Therefore there are still relations which have to be factored out. We do this by inductively repeating the construction $X \rightsquigarrow X_H$ and eventually obtain a sequence
$$ X \to X_H \to X_{H'} \to X_{H''} \ldots $$
Now Let $X_0^f$ be the colimit 
$$ X_0^f := \indlim(X \to X_H \to X_{H'} \to X_{H''} \ldots).$$
If $f$ is a monomorphism then the morphism $X \to X_0^f$ is an isomorphism because nothing has been identified so far.   

\begin{lemma}\label{lemma0}
For each $A_j \to X_0^f$ which is the image of a morphisms $A_j \to Y$ (via the morphism $Y \stackrel{f}{\to} X \to X_0^f$) we have a distinguished filler $B_j \to X_0^f$ such that $f$ sends distinguished fillers to distinguished fillers. \\
Furthermore let $\varphi: X \to Z$ be a morphism in $\C$, where $Z$ is an algebraic fibrant object. If the composition $Y \to X \to Z$ is a morphism of algebraic fibrant objects, then $\varphi$ factors uniquely through $X_0^f$:
$$
\xymatrix{
Y \ar[d]_f \ar[rd]^{\text{alg.}} & \\
X \ar[d] \ar[r]^{\varphi} & Z \\
X^{f}_{0} \ar@{-->}[ru]_{\exists !  \varphi_0} & 
}
$$
\end{lemma}
\begin{proof}
The universal property holds by construction, because the images of distinguished fillers in $Z$ have to be equal due to the fact that $Y \to X \to Z$ is a morphism of algebraic fibrant objects. 
\end{proof}

The statement of lemma \ref{lemma0} looks very similar to proposition \ref{fact}. The only difference is that $X_0^f$ is not necessarily an algebraic fibrant objects because we only have fillers for those morphisms $A_j \to X_0^f$ which factor through $Y$. Hence we will attach fillers for the other morphisms in a free way, like we did in the first section. Let $X_1^f$ be the pushout of the diagram
$$
\xymatrix{
\bigsqcup A_j \ar[r] \ar[d] & X_0^f \ar[d] \\
\bigsqcup B_j \ar[r] & X_1^f
}
$$
where the disjoint union is taken over all $ A_j \to X_0^f$ which do not factor through $Y$. Then we proceed exactly as in the first section and obtain an algebraic fibrant object
$$ X_\infty^f := \indlim(X_0^f \to X_1^f \to X_2^f \tt \ldots ).$$
The inclusion $X_0^f \to X_\infty^f$ is a trivial cofibration. Thus in the case that $f$ is a monomorphism the composition $X \to X_\infty^f$ is also a trivial cofibration. The universal property of $X_\infty^f$ stated in proposition \ref{fact} holds by lemma \ref{lemma0} and construction of $X_\infty^f$. Hence we have proven proposition \ref{fact}. \\

Following an observation of Mike Shulman, we have the slightly more general statement:
\begin{prop}\label{fact2}
The functor $U: \AlgC \to \C$ is solid (or semi-topological). That means, for every family 
$$ \big\{ f_i: U(Y_i) \to X \big\}_{i \in I}$$
where $Y_i$ are algebraic fibrant objects and $X \in \C$ there exists a semi-final lift.
 That is an object $X^{f}_\infty$ together with a morphism $X \to U(X^{f}_\infty)$ such that
\begin{enumerate}
\item all the morphisms $U(Y_i) \to U(X^{f}_\infty)$ are in $\AlgC$;
\item $X^{f}_\infty$ is universal (initial) with this property.
\end{enumerate}
\end{prop}
\begin{proof}
Take the proof of proposition \ref{fact} and replace the phrase ``...which factors through $Y$...'' by the new phrase ``...which factors through one of the $Y_i$'s...''.
\end{proof}
\subsection{Limits and Colimits in \AlgCH} \label{sec:colimits}

In order to show that $\AlgC$ is a model category we will show that finite limits and small colimits exist. Furthermore a precise understanding of pushouts is needed to construct the model structure. Thus in this section we want to investigate limits and colimits in $\AlgC$. \\

Let us start with limits because they are easy to understand. Consider a diagram 
$$F: \mathcal D \to \AlgC$$
indexed by a small category $\mathcal C$. It is easy to check that $\lim F$ is computed as the limit of the underlying diagram in $\C$. In particular that limit exists if and only if the limit in $\C$ exists. Since $\C$ is a model category by definition finite limits exist and so they also exist in $\AlgC$. \\

Unfortunately arbitrary colimits in $\AlgC$ are not so simple. Therefore we start with the special case of filtered colimits. Consider the filtered diagram $(L_\alpha)_{\alpha \in I}$ where $I$ is a well-ordered set. Then let 
$$
L:= \indlim\big( L_\alpha \mid \alpha \in I \big)
$$
be the colimit of the underlying objects of $\C$. A morphism $A_j \to L$ factors through a finite step  $L_{\alpha_0}$ because $A_j$ is small, thus we have a filler $B_j \to L_{\alpha_0} \to L$. Note that $\alpha_0$ and the morphism $A_j \to L_{\alpha_0}$ are not uniquely determined by the morphism $A_j \to L$, but the filler $B_j \to A$ is. This turns $L$ into an algebraic fibrant object which is the colimit over $(L_\alpha)_{\alpha \in I}$.\\

So far we have shown:
\begin{prop} \label{filtered}
Limits and filtered colimits in $\AlgC$ exist and are computed as the limits resp. filtered colimits of the underlying objects of $\C$. That means that the functor $U:\AlgC \to \C$ preserves limits and filtered colimits.
\end{prop}

\begin{corollary}\label{presentable}
If the category $\C$ is locally presentable then $\AlgC$ is also locally presentable.
\end{corollary}
\begin{proof}
By proposition \ref{monad} we know that $\AlgC$ is the category of $T$-algebras for the monad $T = U \circ F$. The category of algebras over a monad is locally presentable if $\C$ is locally presentable and the monad is accessible, i.e. $T$ is accessible (see \cite{AR94} for the definitions and statements). In order to show this we have to show that $T$ preserves filtered colimits. But this follows from proposition \ref{filtered} and the fact that $U$ is left adjoint, hence preserves all colimits.
\end{proof}

Finally a general colimit over a diagram $F: \mathcal D \to \AlgC$ can be described using the auxiliary construction of the last section.
Let 
$$ X := \colim_{d \in \mathcal D} UF(d)$$
be the object of $\C$ which is the colimit of the underlying objects $UF(d)$. There is a morphism 
$$f_d: UF(d) \to X$$
in $\C$ for each $d \in \mathcal D$ given by the colimit inclusions. Thus we can apply proposition \ref{fact2} and obtain an algebraic fibrant object $X_\infty^f$ together with a morphism
$$
F(d) \to X \to X_\infty^f
$$
of algebraic fibrant objects.

\begin{prop}
The algebraic fibrant object $X_\infty^f$ is the colimit over $F$. 
\end{prop}

\begin{proof}
We just check the universal property. Let $\{ \varphi_d: F(d) \to Z \}_{d \in \mathcal D}$
be a family of morphisms in $\AlgC$  such that for each morphism $d \to d'$ in $\mathcal D$ the diagram
\begin{equation}\label{colimDiag}
\xymatrix{
F(d) \ar[r]^{\varphi_d}\ar[d] & Z \\
F(d') \ar[ru]_{\varphi_{d'}} &
}
\end{equation}
commutes. Commutativity of \eqref{colimDiag} implies that the $\varphi$'s factors uniquely through 
$X = \colim_{d \in \mathcal D} UF(d)$ (where the factorization is by morphisms of the underlying objects of $\C$). Thus we can apply proposition \ref{fact2} and obtain a unique morphism
$$ \varphi_\infty: X_\infty^f \to Z$$
of algebraic fibrant objects.
\end{proof}

\begin{corollary}\label{complete}
$\AlgC$ is  complete and cocomplete if $\C$ is.
\end{corollary}

In contrast to the case of filtered colimits the morphism $X \to X_\infty^f$ is in general not an isomorphism or weak equivalence. This means that colimits in $\AlgC$ cannot be computed as colimits of the underlying simplicial sets, not even up to weak equivalence. Nevertheless we can simplify the construction for pushouts along free objects. \\

In more detail: let $i: A \to B$ be a morphism in $\C$, $Y$ be an algebraic fibrant object and consider a diagram
\begin{equation}\label{push}
\xymatrix{
FA \ar[r]\ar[d]_{Fi} & Y \\
FB &
}
\end{equation}
in $\AlgC$. We will give a simple description of the pushout of this diagram. First such a diagram is the same as a diagram
$$
\xymatrix{
A \ar[r]\ar[d]_{i} & U(Y) \\
B &
}
$$
in $\C$ by adjointness of $F$ and $U$. Let 
$$X := Y\cup_A B \in \C$$
be the pushout of the last diagram.  This comes with a morphism $f: U(Y) \to X$. 
We apply proposition \ref{fact} to this morphism and obtain an algebraic fibrant object
$$X_\infty^f = (Y\cup_A B)_\infty^f $$
together with a morphism $X \to X_\infty^f$. 

\begin{prop}\label{glue}
The algebraic fibrant object $(Y\cup_A B)_\infty^f$ is the pushout of diagram \eqref{push}. 
If $i: A \to B$ is a trivial cofibration then  
$$Y \to (Y\cup_A B)_\infty^f$$  
is also a trivial cofibration.
\end{prop}
\begin{proof}
We first check that $(Y\cup_A B)_\infty^f$ satisfies the universal property of the pushout. Therefore let $Z$ be an algebraic fibrant object. A morphism $(Y\cup_A B)_\infty^f$ to $Z$ is then by proposition \ref{fact} the same as a morphism $g: Y\cup_A B \to Z$ in $\C$, such that the composition 
$$ Y \stackrel{f}{\to} Y\cup_A B \to Z $$
preserves distinguished fillers, i.e. is a morphism of algebraic fibrant objects. But such a morphism $g$ is the same as a morphism of algebraic fibrant objects $g_1: Y \to Z$ and a morphism $g_2: B \to U(Z)$ in $\C$ which agree on $A$. The adjunction $(F,U)$ implies that $g_2$ is the same as a morphism $FB \to Z$ which completes the proof that $(Y\cup_A B)_\infty^f$ is the pushout of diagram \eqref{push}.\\
It remains to show that $Y \to (Y\cup_A B)_\infty^f$ is a trivial cofibration if $i: A \to B$ is one. This morphism is the composition 
$$ Y \stackrel{f}\to Y\cup_A B \to (Y\cup_A B)_\infty^f. $$
The first morphism $f$ is a trivial cofibration, because it is the pushout of the trivial cofibrations $i$. Hence our general assumption implies that $f$ is monic. Thus the last part of \ref{fact} show that also
$$ Y\cup_A B \to (Y\cup_A B)_\infty^f $$
is a trivial cofibration. Hence also the composition of those two morphisms is a trivial cofibration, which completes the proof.
\end{proof}
\begin{remark}\label{remarkmono}
Assume the category $\C$ also satisfies that all cofibrations are monic. Then the same argument shows that for 
a cofibration $i: A \to B$ the morphism
$$Y \to (Y\cup_A B)_\infty^f$$  
is also a cofibration.
\end{remark}

\subsection{Model structure on \AlgCH}

In this section we want to endow $\AlgC$ with a model structure, such that the pair $(F,U)$ of functors is a Quillen equivalence. 

\begin{definition}\label{fib}
A morphism $f: X \to Y$ of algebraic fibrant objects is a weak equivalence (fibration) if the underlying morphism $U(f): U(X) \to (Y)$ is a weak equivalence (fibration) in $C$. A morphism is a cofibration of algebraic fibrant objects, if it has the LLP with respect to trivial fibrations.
\end{definition}

In order to show that this yields a model structure on $\AlgC$ we first recall that $\AlgC$ is finite complete and small cocomplete. We want to use the general transfer principle for cofibrantly generated model structures, see \cite{crans1995quillen}. 
Therefore we have to show that
\begin{enumerate}
\item\label{eins} the functor $F$ preserves small objects; 
\item\label{zwei} relative $F(I)$-cell complexes are weak equivalences in $\AlgC$, where 
$$ I:= \big\{C_i \to D_i\} $$
is the original set of generating trivial cofibrations.
\end{enumerate}
Condition \ref{eins} holds if the right adjoint $U$ preserves filtered colimits which is true in our case due to proposition \ref{filtered}. For the second condition recall that a $F(I)$-cell complex is a transfinite composition of pushouts of the form
\begin{equation*}
\xymatrix{
FC_i \ar[r]\ar[d] & Y \ar[d]_-g  \\
FD_i \ar[r]    &    \big(Y\cup_{C_i} D_i\big)_\infty^f     
}
\end{equation*}
Proposition \ref{glue} implies that $g$ is a trivial cofibration of the underlying objects of $\C$. Transfinite composition of morphisms commutes with the forgetful functor $U$ since $U$ preserves filtered colimits (see proposition \ref{filtered}). Thus the fact that a transfinite composition of trivial cofibrations in $\C$ is again a trivial cofibration shows that condition \ref{zwei} also holds. Moreover we have shown:

\begin{corollary}\label{preserves}
The functor $U$ preserves trivial cofibrations. If additionally in $\C$ all cofibrations are monic, then $U$ also preserves cofibrations. 
\end{corollary}
\begin{proof}
The first part follows from the discussion above. Using remark \ref{remarkmono} a similar argument shows that an $F(I')$-cell complex is a cofibration where $I'$ is the set of generating cofibrations.
\end{proof}

\begin{lemma}
Unit and counit of the adjunction $(U,F)$ are weak equivalences. Thus a morphism $\varphi: X \to Y$ in $\C$ is a weak equivalence iff $F(\varphi): FX \to FY$ is a weak equivalence.
\end{lemma}
\begin{proof}
The unit of the adjunction is the morphism $X \to X_\infty$ which is by construction a fibrant replacement in $\C$, hence a weak equivalence. The counit is a morphism $m: FA \to A$ for an algebraic fibrant object which fits into the diagram
$$
\xymatrix{
A \ar[r]\ar[rd]_{id} &  FA \ar[d]^m \\
& A 
}
$$
Hence by 2-out-of-3 this is also weak equivalence.
\end{proof}

Altogether we have proven the main theorem of this section:
\begin{thm}\label{main}
The category $\AlgC$ admits a closed Quillen model structure with fibrations, weak equivalences and cofibrations as in definition \ref{fib}. The generating (trivial) cofibrations are the images of the generating (trivial) cofibrations under $F$. The pair of adjoint functors
$$ 
\xymatrix{
F : \C \ar@<0.3ex>[r] & \AlgC : U \ar@<0.7ex>[l] 
}$$
is a Quillen equivalence.
\end{thm}

\begin{corollary}
Each cofibrantly generated model category $\C$ where all trivial cofibrations are monic is Quillen equivalent to a model category where all objects are fibrant. In particular $\C$ is Quillen equivalent to a right proper model category.
\end{corollary}
\begin{proof}
It is clear from the definition that all objects in $\AlgC$ are fibrant. But each model category where all objects are fibrant is right proper.
\end{proof}

As a consequence of corollary \ref{presentable} we have
\begin{proposition}\label{combinatorial}
If $\C$  is a combinatorial model category then so is $\AlgC$.
\end{proposition}

Finally we want draw another consequence
\begin{proposition}\label{fibmon}
Let $\C$ be a cofibrantly generated model category where all cofibrations are monic. Then there is an endofunctor $T: \C \to \C$ such that 
\begin{itemize}\setlength{\itemsep}{-0.5ex}
\item $T(X)$ is a fibrant replacement of $X$.
\item $T$ preserves filtered colimits.
\item $T$ preserves cofibrations and trivial cofibrations.
\item $T$ is monadic.
\end{itemize}
\end{proposition}
\begin{proof}
The functor $T := U \circ F$ is monadic by proposition \ref{monad}. It preserves filtered colimits since $F$ preserves all colimits and $U$ preserves filtered colimits by \ref{filtered}. Analogously $F$ preserves cofibrations and trivial cofibrations since it is left Quillen and $U$ preserves them by corollary \ref{preserves}. Hence $T$ preserves cofibrations and trivial cofibrations.
\end{proof}

\section{Algebraic Kan complexes}\label{algKan2}

In this section we want to apply the general construction of the last section to the case of the standard (Quillen) model structure on simplicial sets. That leads us to the notion of algebraic Kan complex. We will explain how this should be considered as an algebraic notion of an $\infty$-groupoid and give a direct comparison to topological spaces.

\subsection{Kan complexes and \inftyH -groupoids}\label{sec:kan}

Let $\sSet$ denote the category of simplicial sets. This category carries the structure of a cofibrantly generated model category where the generating trivial cofibrations are given by the horn inclusions 
$$ \{\Lambda^k(n) \to \Delta(n) \mid n \geq 1, 0 \leq k \leq n\}$$
and the generating cofibrations are given by the boundary inclusions
$$ \{\partial \Delta(n) \to \Delta(n) \mid n \geq 1\}.$$
For these statements and terminology see \cite{hovey2007model} chapter 3.
The fibrant objects, i.e. the simplicial sets $X$ having fillers for all horns $\Lambda^k(n) \to X$ are called Kan complexes. It is well known that Kan complexes could be seen as a model for $\infty$-groupoids 
(i.e. weak $\infty$-categories where all morphisms are invertible). As an illustration we will investigate the lifting conditions for the horns of the 2-simplex $\Delta(2)$.
There are three horns $\Lambda^0(2), \Lambda^1(2)$ and $\Lambda^2(2)$. First consider the inner horn $\Lambda^1(2)$. A morphism 
$$h: \Lambda^1(2) \to X $$ 
is the same as choosing two matching 1-cells $f,g \in X_1$, depicted as 
$$\xymatrix{
& b\ar[rd]^g &  \\
a \ar[ru]^f && c
}$$
Now a filler $F(h): \Delta(2) \to X$ is a 2-simplex that fills this horn:
$$
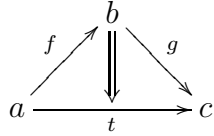
\xymatrix{
& b\ar[rd]^g \ar@{=>}[d]&  \\
a \ar[ru]^f\ar[rr]_t && c
}$$
The target $t$ of this 2-cell should now be seen as a composition of $f$ and $g$. Of course the filler $F(h)$ is not unique and thus the composition of 1-cells is also not unique. Nevertheless, using the higher dimensional fillers one can show that composition defined in this way is unique up two 2-cells (that means between two composites there is always a 2-cell connecting them, which is also unique up to 3-cells...). But the lack of a fixed compositions is sometimes counterintuitive or might lead to problems working with $\infty$-groupoids. Thus the idea is to fix a filler for each pair of morphisms $(f,g)$ and refer to this as \emph{the composition} of $f$ and $g$. We give a definition of Kan complexes with fixed fillers, called algebraic Kan complexes, in the next section. \\

But let us first return to the investigation of lifting properties. We saw that the lifting against the inner horn $\Lambda^1(2)$
endows $X$ with compositions of 1-cells. Analogously one can see that lifting against higher inner horns $\Lambda^k(n)$ provides compositions of higher cells, which is a good exercise to do for $n = 3$. But we want to look at the outer horns $\Lambda^0(2)$ and $\Lambda^2(2)$. A morphism $\Lambda^0(2) \to X$ provides two morphisms $f,g \in X_1$ that fit together like this:
$$
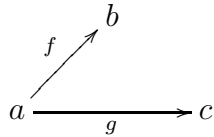
\xymatrix{
& b &  \\
a \ar[ru]^f \ar[rr]_g && c
}$$
A filler for such a diagram translates into a diagram
$$
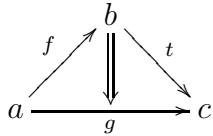
\xymatrix{
& b\ar[rd]^t\ar@{=>}[d] &  \\
a \ar[ru]^f \ar[rr]_g && c
}.$$
This means that $g$ could be seen as a composition $t \circ f$ or equivalently $t = g \circ f^{-1}$. In that way a Kan complex provides inverses an thus models $\infty$-groupoids rather than $\infty$-categories.  In our approach to $\infty$-groupoids we will also fix those fillers and thus have a choice of inverses. Later in section \ref{quasi} we will consider quasi-categories where fillers are only required for inner horns and thus there are no inverses for 1-cells.

\subsection{Algebraic Kan complexes as \inftyH -groupoids}

In this section we will give the notion of algebraic Kan complex and use the general methods developed in section \ref{allgemein} to obtain a model structure and to deduce properties for algebraic Kan complexes. As motivated in the last section, an algebraic Kan complex should have fixed fillers for all horns, and thus fixed compositions and inverses of cells.

\begin{definition} \phantomsection\label{algKan}
\begin{enumerate}
\item
An  \emph{algebraic Kan complex} is a simplicial set $X$ together with a distinguished filler for each horn in $X$. A map of algebraic Kan complexes is a map that sends distinguished fillers to distinguished fillers. The category of algebraic Kan complexes is denoted by \Alg 
\item
A morphism $f: X \to Y$ of algebraic Kan complexes is a weak equivalence (fibration) if the underlying morphism $U_A(f): U_A(X) \to U_A(Y)$ is a weak equivalence (fibration) of simplicial sets. A morphism is a cofibration of algebraic Kan complexes, if it has the LLP with respect to trivial fibrations.
\end{enumerate}
\end{definition}

The model category $\sSet$ is cofibrantly generated and the cofibrations are exactly the monomorphisms. Thus from section \ref{allgemein} we immediately have:

\begin{thm}\label{algKanMain}
\begin{enumerate}
\item The canonical forgetful functor $U_A: \Alg \to \sSet$ has a left adjoint $F_A: \sSet \to \Alg$ which is constructed by iteratively attaching $n$-simplices as fillers for all horns. 
\item Algebraic Kan complexes are precisely algebras for the monad $T:=U \circ F$ generated by this adjunction. 
\item \Alg is small complete and cocomplete. Limits and filtered colimits are computed as limits resp. colimits of the underlying simplicial sets. 
\item \Alg is a combinatorial model category with generating trivial cofibrations 
$$ F_A\Lambda^k(n) \to F_A\Delta(n) $$
and generating cofibrations 
$$ F_A\partial\Delta(n) \to F_A \Delta(n)$$
\item The pair $(F,U)$ is a Quillen equivalence. Furthermore the functor $U_A$ preserves cofibrations and trivial cofibrations.
\end{enumerate}
\end{thm}
\begin{proof}
1: prop. \ref{adjoint};   2: prop. \ref{monad}; 3: prop \ref{filtered} and corollary \ref{complete};  4: theorem \ref{main} and proposition \ref{combinatorial}; 5: theorem \ref{main} and corollary \ref{preserves}.
\end{proof}

Note that in contrast to $\sSet$ in this model structure on $\Alg$ each object is fibrant but not necessarily cofibrant. For example the point in $\Alg$ is not cofibrant. The cofibrant objects are exactly retracts of $F_A \partial\Delta(n) \to F_A \Delta(n)$-cell complexes. We will say some words about such cell complexes in order to give a better understanding of the cofibrations.  
Let $X$ be an algebraic Kan complex. We want to glue a n-cell to $X$ along its boundary $\partial \Delta(n)$. Formally speaking we want to compute the pushout of a diagram
$$\xymatrix{
F_A \partial \Delta(n) \ar[r]\ar[d] & X \\
F_A \Delta(n) &
}$$
where the upper morphism comes from a morphism $\partial\Delta(n) \to X$ 
of simplicial sets, which is just a combinatorial $n$-sphere in $X$. From prop \ref{glue} we know that the pushout can now be computed in two steps: first glue the $n$-cell along its boundary to $X$, i.e. form the pushout $X \cup_{\partial\Delta(n)} \Delta(n)$ in \sSet. Intuitively speaking we simply add a new $n$-cell to our $\infty$-groupoid. But now some compositions are missing, namely those of the new $n$-cell with cells of the old $\infty$-groupoid $X$. Thus we throw in freely all those compositions, i.e. form $X_\infty^f$ (see section \ref{sec:aux}). What we finally obtain is the pushout in the category $\Alg$. \\

Note first that gluing a $n$-cell not along its boundary, but along its horn works exactly the same. Now general cell complexes are just an iteration of this gluing process. The fact that filtered colimits are computed as colimits of the underlying simplicial sets means, that we can do this iteration naively and finally obtain the right algebraic Kan complex. Hence we have a very clear understanding of cofibrations and trivial cofibrations in $\Alg$. This discussion also shows that the category $\Alg$ provides the right colimits, whereas colimits of (ordinary) Kan complexes might no longer be Kan complexes and thus are not the correct colimit of $\infty$-groupoids.

\subsection{The homotopy hypothesis}

The homotopy hypothesis is informally speaking the idea that $\infty$-groupoids are the same as topological spaces. Here we propose algebraic Kan complexes as a model for $\infty$-groupoids. Therefore we should show that they are equivalent to topological spaces. More precisely we want to prove that the model categories are Quillen equivalent. As model categories are a way to encode the $(\infty,1)$-category of $\infty$-groupoids, this could be regarded as proving the homotopy hypothesis for our model of $\infty$-groupoids. \\

First of all, it is a classical result of Quillen, that the (standard) model categories of topological spaces and simplicial sets are equivalent. The adjoint functors which form the Quillen pair are
$$ 
\xymatrix{
|...| : \sSet \ar@<0.3ex>[r] & \Top : \Sing \ar@<0.7ex>[l] 
}$$
where the left adjoint $|...|$ is the geometric realization functor and the right adjoint $\Sing$ is the singular complex functor. We could now argue, using this result, that the category of algebraic Kan complexes is Quillen equivalent to simplicial sets and thus is equivalent to topological spaces. This is perfectly fine on the level of $(\infty,1)$-categories. But in this way we will not obtain a direct Quillen equivalence between algebraic Kan complexes and topological spaces, because the Quillen equivalences cannot be composed. Instead we will give a direct Quillen equivalence
$$ 
\xymatrix{
|...|_r : \Alg \ar@<0.3ex>[r] & \Top : \ASing \ar@<0.7ex>[l] 
}$$
where the left adjoint $|...|_r$ is called reduced geometric realization and the right adjoint $\ASing$ is called fundamental $\infty$-groupoid. This will render the diagram
\begin{equation}\label{triangle}
\xymatrix{
& \Alg\ar@<0.3ex>[rdd]^-{|...|_r}\ar@<0.7ex>[ldd]^-{U_A} & \\
& & \\
\sSet \ar@<-0.7ex>[rr]_-{|...|}\ar@<0.3ex>[ruu]^-{F_A}&  &\Top \ar@<-0.3ex>[ll]_-{\Sing} \ar@<0.7ex>[luu]^-{\ASing} 
}\end{equation}
commutative (more precisely: the inner and the outer triangle). \\

Let us start by describing the fundamental $\infty$-functor $\ASing: \Top \to \Alg$. For a topological space $M$ the ordinary singular complex is by definition the simplicial set $\Sing(M)$ with 
$$ \Sing(M)_n = \hom_{\Top}\big( |\Delta(n)| , M\big) $$
where $| \Delta(n) |$ denotes the topological $n$-simplex 
$$ |\Delta(n)| = \big\{(x_0,\ldots,x_n) \in \mathbb{R}_{\geq 0}^{n+1} \mid \sum x_i = 1\big\}.$$ 
In order to make the diagram $\eqref{triangle}$ commutative, the underlying simplicial set of $\ASing(M)$ has to be the simplicial set $\Sing(M)$. Now to endow $\Sing(M)$ with the structure of an algebraic Kan complex, we have to give distinguished fillers for all horns
$$ \Lambda^k(n) \to \Sing(M).$$
But due to the fact that $\Sing$ is right adjoint to $|...|$ such a horn is the same as a morphism 
$$ h: |\Lambda^k(n)| \to M$$
of topological spaces. It is easy to see that $|\Lambda^k(n)|$ is (homeomorphic to) the naive horn
$$ |\Lambda^k(n)| = \bigcup_{i \neq k} \big\{(x_0,\ldots,x_n) \in |\Delta(n)|  \mid x_i = 0\big\} $$
which is the union of all but one faces of the simplex $|\Delta(n)|$. From the geometric point of view, it is clear that there are (linear) retractions 
$$R(n,k):  |\Delta(n)| \to |\Lambda^k(n)|.$$
We will not give an explicit formula for the $R(n,k)$ because that will not give more insights, but in principle that can be easily done. We use these retractions to obtain morphisms
$$ |\Delta(n)| \stackrel{R(n,k)}{\longrightarrow} |\Lambda^k(n)| \stackrel{h} \to M$$
which by adjointness are fillers $\Delta(n) \to \Sing(M)$ for horns in $\Sing(M)$. We denote the resulting algebraic Kan complex by $\ASing(M)$. Furthermore this assignment is obviously functorial in $M$ such that we finally have defined the functor
$$\ASing: \Top \to \Alg$$
\begin{remark}
The construction of the functor $\ASing$ depends on the choice of retracts $R(n,k)$ we have made. Every other choice would lead to a different (but of course weakly equivalent) algebraic Kan complex $\ASing(M)$. This choice parameterizes the composition of paths or higher cells in the path $\infty$-groupoid. 
\end{remark}

Now let us turn towards the reduced geometric realization functor
$$ |\quad|_r: \Alg \to \Top.$$
So let $X$ be an algebraic Kan complex. First of all, consider the geometric realization 
$ |U_A(X)|$ of the underlying simplicial set. The distinguished fillers $\Delta(n) \to X$ provide 
n-simplices in $|U_A(X)|$ which are fillers for the horns $|\Lambda^k(n)| \to |U_A(X)|$. But the composite
$$|\Delta(n)| \stackrel{R(n,k)}{\longrightarrow} |\Lambda^k(n)| \to |U_A(X)|$$
provides another filler. Therefore we define the reduced geometric realization $|X|_r$ as the space
where those two different fillers for the same horn have been identified. Formally we have
$$ |X|_r := \CoEq\Big(\bigsqcup |\Delta(n)| \rightrightarrows |U_A(X)|\Big) $$
where the disjoint union is taken over all horns $\Lambda^k(n) \to X$. With this definition we have:

\begin{prop}
The functor $|\quad|_r : \Alg \to \Top$ is left adjoint to $\ASing$.
\end{prop}
\begin{proof}
Let $X$ be an algebraic Kan complex, $M$ a topological
 space and $f: |X|_r \to M$ be a continuous map. By construction of $|X|_r$ as a coequalizer this is the same as a continuous map 
$\tilde f:|U_A(X)| \to M$ such that for each horn $h: \Lambda^k(n)\to X$ with filler $F(h): \Delta(n) \to X$ the two maps 

$$ |\Delta(n)| \stackrel{R(n,k)}{\longrightarrow} |\Lambda^k(n)| \stackrel{|h|}{\longrightarrow} |U_A(X)|\stackrel{\tilde f}{\to} M$$
and 
$$ |\Delta(n)| \stackrel{|F(h)|}{\longrightarrow} |U_A(X)|\stackrel{\tilde f}{\to} M$$
agree. Using the adjunction $(|...|,\Sing)$ we see that this is the same as a morphism 
$$ \tilde{\tilde f}: U_A(X) \to \Sing(M)$$
such that the images of distinguished filler diagrams in $X$ are sent to the fillers in $\Sing(M)$ obtained by using the retractions $R(n,k)$. 
That means $\tilde{\tilde f}$ is a morphism of algebraic Kan complexes between $X$ and $\ASing(M)$. 
\end{proof}

\begin{corollary}
Diagram \eqref{triangle} commutes (up to natural isomorphism).
\end{corollary}
\begin{proof}
The inner triangle commutes by construction of $\ASing$. For commutativity of the outer we have to show that $|...|_r \circ F_A \cong |...|$. From the fact that $|...|_r$ is left adjoint to $\ASing$ and $F_A$ is left adjoint to $U_A$ we deduce that $|...|_r \circ F_A$ is left adjoint to $U_A \circ \ASing$. By commutativity of the inner triangle the latter is equal to $\Sing$. That means that  $|...|_r \circ F_A$ and $|...|$ are left adjoint to $\Sing$ and thus are naturally isomorphic.
\end{proof}
\begin{corollary}\label{homhyp}
The pair $\big(|\quad|_r,\ASing\big)$ is a Quillen equivalence.
\end{corollary}
\begin{proof}
We already know that $(|...|,\Sing)$ and $(F_A,U_A)$ are Quillen equivalences. By the 2-out-of-3 property for Quillen equivalences it follows that $(|...|_r,\ASing)$ is also a Quillen equivalence.
\end{proof}

\section{Algebraic quasi-categories}\label{quasi}

In this section we want to apply the general principle to the Joyal model structure on simplicial sets. Thereby we are led to introduce the concept of algebraic quasi-category as an algebraic model for $(\infty,1)$-categories. Finally we will relate algebraic quasi-categories to algebraic Kan complexes.

\subsection{Quasi-categories as (\inftyH,1)-categories}\label{sec:quasi}

The category $\sSet$ carries another model structure besides the Quillen model structure (see \cite{JoyalBarcelona}, \cite{lurie2006higher}). This second model structure is called the Joyal model structure. Unfortunately it is more complicated than the Quillen structure, but it is also cofibrantly generated. The cofibrations are the same as in the Quillen structure and thus the boundary inclusions
$$ \partial \Delta(n) \to \Delta(n). $$
are a set of generating cofibrations. But there is no known description of a set of generating trivial cofibrations, although it is known that such a set exists. The weak equivalences in this model structure are called categorical equivalences or quasi-equivalences. \\

The fibrant objects in this model structure are called quasi-categories. These are the simplicial sets $X$ which have the right lifting property against all inner horns
$$ \Lambda^k(n) \to \Delta(n) $$
for $n \geq 2$ and $0 < k < n$. We described in section \ref{sec:kan} how these lifting conditions could be seen as providing compositions of cells. The fact that we only have lifting conditions against inner horns thus means that we do not have inverses to 1-cells. A more precise treatment of these lifting properties shows that we still have inverses for $n$-cells with $n \geq 2$. That means that quasi-categories are a model for $(\infty,1)$-categories, i.e. $\infty$-categories where all $n$-morphisms for $n \geq 2$ are invertible. And in fact there has been much work providing evidence that this is an appropriate model for $(\infty,1)$-categories. See \cite{bergner2006survey} for a good introduction. \\

But as in the case of Kan complexes it is desirable to have a more algebraic model where especially compositions of morphisms are not only guaranteed to exist but are specified. We will do this by applying our general construction from section \ref{allgemein} to quasi-categories, as we did for Kan complexes in section \ref{algKan2}. 

\subsection{Algebraic quasi-categories}\label{algqdef}

Let $J$ be the set of inner horn inclusions
$$
\Lambda^k(n) \to \Delta(n)
$$
which are trivial cofibrations in the Joyal model structure. Using this set of morphisms we follow the general pattern from section \ref{allgemein}:
\begin{definition} 
An \emph{algebraic quasi-category} is a simplicial set $X$ together with a distinguished filler for each inner horn in $X$. A map of quasi-categories is a map that sends distinguished fillers to distinguished fillers. We denote the category of algebraic quasi-categories by \AlgQ.
\end{definition}
\begin{thm}
\begin{enumerate}
\item The canonical forgetful functor $U_Q: \AlgQ \to \sSet$ obtains a left adjoint $F_Q: \sSet \to \AlgQ$ which is constructed by iteratively attaching $n$-simplices as fillers for all inner horns. 
\item Algebraic quasi-categories are algebras for the monad $T_Q:=U_Q \circ F_Q$ generated by this adjunction. 
\item \Alg is small complete and cocomplete. Limits and filtered colimits are computed as limits resp. colimits of the underlying simplicial sets. 
\end{enumerate}
\end{thm}
\begin{proof}
1: theorem \ref{adjoint};   2: theorem \ref{monad}; 3: theorem \ref{filtered} and corollary \ref{complete}.
\end{proof}
Additionally we have the model structure on $\AlgQ$:
\begin{definition}\label{modelQuasi}
A morphism $f: X \to Y$ of algebraic quasi-categories is a weak equivalence (fibration) if the underlying morphism $U_Q(f): U(X) \to (Y)$ is a categorical equivalence (fibration) in the Joyal model structure. A morphism is a cofibration of algebraic quasi-categories, if it has the LLP with respect to trivial fibrations.
\end{definition}

Now according to theorem \ref{main} this defines a cofibrantly generated model structure on $\AlgQ$. One of the major advantages of this new model structure is that we can explicitly write down a set of generating trivial cofibrations. This follows from the fact that in the Joyal model structure a morphism between fibrant objects, i.e. quasi-categories, is a fibration iff it has the LLP with respect to inner horn inclusions
$$ \Lambda^k(n) \to \Delta(n) $$
and the inclusion 
$$ pt \to I$$
of an object in the interval groupoid $I$. By definition $I$ is the nerve of the groupoid with two objects and an isomorphism between them (see \cite{JoyalBarcelona}, Prop. 4.3.2). Thus we have:
\begin{thm}\label{AlgQMain}
\item $\AlgQ$ is a combinatorial model category with generating trivial cofibrations 
$$ F_Q\Lambda^k(n) \to F_Q\Delta(n) \qquad \text{ for }1 < k < n \qquad\qquad  F_Q pt \to F_QI$$
and generating cofibrations 
$$ F_Q\partial\Delta(n) \to F_Q \Delta(n)$$
\item The pair $(F_Q,U_Q)$ is a Quillen equivalence. Furthermore the functor $U_Q$ preserves cofibrations and trivial cofibrations.
\end{thm}
\begin{proof}
Theorem \ref{main} and proposition \ref{combinatorial} show that $\AlgQ$ is a combinatorial model category with the given  generating cofibrations and that the pair $(F_Q,U_Q)$ is a Quillen equivalence. From corollary \ref{preserves} we know that $U_Q$ preserves trivial cofibrations. It only remains to show that the given set of morphisms is a set of generating trivial cofibrations. \\
We show that $f:X \to Y$ is a fibration in $\AlgQ$ if it has the RLP with respect to the given morphisms. By definition $f$ is a fibration iff $U_Q(f)$ is a fibration in the Joyal model structure. Since $U_Q(X)$ and $U_Q(Y)$ are quasi-categories, this is the case if $U_Q(f)$ has the RLP with respect to inner horn inclusions and $pt \to I$. Using the fact that $F_Q$ is left adjoint to $U_Q$ we see that $f$ is a fibration in $\AlgQ$ iff it has the RLP with respect to the given set of morphisms.
\end{proof}

\subsection{Groupoidification}

In this section we want to investigate how the (model) categories $\Alg$ and $\AlgQ$ are related to each other. Remember that objects in both of them are simplicial sets with extra structure. In the case of $\Alg$ we have fixed fillers for all horn inclusions and in the case of $\AlgQ$ we only have fixed fillers for inner horn inclusion. This shows that we have a canonical forgetful functor
$$ V: \Alg\to \AlgQ$$
which forgets the fillers for the outer horns. We will construct a left adjoint 
$$G: \AlgQ \to \Alg$$
called \emph{groupoidification} and show that the pair $(G,V)$ forms a Quillen adjunction (not a Quillen equivalence!). This is the algebraic analogue of the fact that the Quillen model structure on $\sSet$ is a left Bousfield localization of the Joyal model structure. More precisely we have a commuting square

\begin{equation}\label{square}
\xymatrix{
\Alg\ar@<0.7ex>[rr]^-{V} \ar@<-0.7ex>[d]_{U_A}              &&\AlgQ \ar@<0.3ex>[ll]^-{G} \ar@<0.3ex>[d]^{U_Q} \\
\sSet_Q \ar@<-0.7ex>[rr]_-{Id} \ar@<-0.3ex>[u]_{F_A}           && \sSet_J \ar@<-0.3ex>[ll]_-{Id} 
\ar@<0.7ex>[u]^{F_Q}
}
\end{equation}
of Quillen adjunctions, where $\sSet_J$ denotes the category of simplicial sets with the Joyal model structure and $\sSet_Q$ with the Quillen model structure. Here the inner and the outer squares commute (up to natural isomorphism).\\

Now let $X$ be an algebraic quasi-category. We already have fixed fillers for inner horns in $X$, i.e. morphisms $\Lambda^k(n) \to X$ 
with $0<k<n$. In order to build an algebraic Kan complex 
out of $X$ we will freely add fillers for outer horns in $X$, 
i.e. morphisms $\Lambda^k(n) \to X$ with $k=0$ or $k=n$. 
The construction is much the same as the construction from section \ref{sec:free} and we only sketch it. Let $X_1$ be the pushout
$$
\xymatrix{
\bigsqcup \Lambda^k(n) \ar[r] \ar[d] & X \ar[d] \\
\bigsqcup \Delta(n) \ar[r] & X_1
}
$$
where the colimit is taken over all outer horns in $X$. The next step $X_2$ is obtained by gluing $n$-cells $\Delta(n)$ along outer horns $\Lambda^k(n) \to X_1$ that do not factor through $X$. We proceed like this and finally put
$$ G(X) := \indlim(X \to X_1 \to X_2 \to \ldots).$$
\begin{prop}
The functor $G: \AlgQ \to \Alg$ is left adjoint to $V$ and the square \eqref{square} commutes.
\end{prop}
\begin{proof}
By definition of $G$ it is clear that it is left adjoint to $V$. In diagram \eqref{square} the commutativity of the outer square is just a trivial statement about the forgetful functors $U_Q, U_A$ and $V$. Commutativity of the inner square means that we have to show that $G \circ F_Q$  and $F_A$ are naturally isomorphic. This follows from the fact that $G \circ F_Q$ and $F_A$ are both left adjoint to $U_Q \circ V = U_A$.
\end{proof}
\begin{prop}
The pair $(G,V)$ is a Quillen adjunction.
\end{prop}
\begin{proof}
It is enough to show that $V$ preserves fibrations and trivial fibrations. Let $f: X \to Y$ be a (trivial) fibration in $\Alg$. We want to show that $V(f) : V(X) \to V(Y)$ is a (trivial) fibration in $\AlgQ$. By definition \ref{modelQuasi} this is the case iff $U_Q(V(f))$ is a Joyal (trivial) fibration in $\sSet_J$. By definition \ref{algKan} we already know that $U_A(f) = U_Q(V(f))$ is a Quillen (trivial) fibration. Thus the claim follows from the fact that a Quillen (trivial) fibration is a Joyal (trivial) fibration.  This is equivalent to the statement that $Id: \sSet_Q \to \sSet_J$ is a right Quillen functor or to the statement that $\sSet_Q$ is a left Bousfield localization of $\sSet_J$.
\end{proof}

\section{Outlook}\label{outlook}

In this section we want to sketch further possible applications of the general theory of algebraic fibrant objects. They always lead to an algebraic version of the structure modelled by the model category.

\begin{itemize}
\item
There is a nice extension of the category of simplicial sets called dendroidal sets, introduced in \cite{MW07}. Furthermore there is also a model structure on dendroidal sets defined by Cisinski and Moerdijk, which extends the Joyal model structure on simplicial sets, see \cite{MC07}. The fibrant objects are called inner Kan complexes and are a model for $(\infty,1)$-operads. The cofibrations in the category are (normal) monomorphisms,  hence we can apply our general result. This leads to an algebraic model for $(\infty,1)$-operads. Similar to the case of algebraic quasi-categories described in section \ref{algqdef} there is an explicit set of generating trivial cofibrations for algebraic dendroidal sets. Dendroidal sets have been introduced to give recursive definitions of weak $n$-categories. It would be interesting to see whether algebraic dendroidal sets could be used to produce algebraic analogues of these constructions.
\item
The model structure of Quillen and Joyal both use simplicial sets to model $\infty$-groupoids and $(\infty,1)$-categories as explained in section \ref{sec:kan} and \ref{sec:quasi}. But simplicial sets can also be used as a model for all weak $\infty$-categories. This goes back to ideas of Street. In order to do so, simplicial sets are equipped with the extra structure of thin elements which allow to keep track of invertible higher cells. The category obtained in this way is called the category of stratified simplicial sets. On this category there is a model category structure constructed by Verity \cite{Ver06}. The fibrant objects are called weak complicial sets. In this model structure the cofibrations are also monomorphisms, hence by applying our general construction this leads to an algebraic version of weak $\infty$-categories. It would be interesting to investigate this model structure in more detail, in particular to see how the nerves of strict $\infty$-categories lead to algebraic weak complicial sets.
\item
Simplicial presheaves over a site $S$ are presheaves with values in the category of simplicial sets. They also carry model structures which exhibits them as models for $\infty$-stacks. The two most important model structures are the local projective and the local injective model structure, see e.g. \cite{Jardine} and \cite{Dugger01}. In both of them the cofibrations are monomorphisms and hence we obtain two categories of algebraic simplicial presheaves. They thus form a model for algebraic $\infty$-stacks. Using results of \cite{DHI04} one can make explicit the descent conditions in these algebraic $\infty$-stacks. This might provide a framework in which some gluing constructions can be made direct and functorial. 
\end{itemize}
There are of course many more possible applications of the general construction. Let us finally note that the fibrant replacement monad $T: \C \to \C$ investigated in Proposition \ref{fibmon} is very useful in some applications, even if we do not want to deal with algebraic fibrant objects. It allows for example to replace diagrams (e.g. homotopy pushout diagrams) nicely in very general situations.

\bibliographystyle{alpha}
\bibliography{AlgKan}

\end{document}